\documentclass[a4paper,11pt]{article}
\usepackage{authblk}
\title{Higher genus theory}

\usepackage{amsmath, amssymb, mathrsfs, amsthm, enumitem}
\usepackage{hyperref}
\usepackage[all]{xy}
\usepackage[top=2.5cm, bottom=2.5cm]{geometry}

\newcommand{\Z}{\mathbb{Z}}

\newcommand{\Q}{\mathbb{Q}}
\newcommand{\FF}{\mathbb{F}}

\newcommand{\bmt}{\begin{pmatrix}}
\newcommand{\emt}{\end{pmatrix}}
\newcommand{\bsm}{\left(\begin{smallmatrix}}
\newcommand{\esm}{\end{smallmatrix}\right)}

\newtheorem{definition}{Definition}[section]

\newtheorem{remark}[definition]{Remark}

\theoremstyle{plain}
\newtheorem{proposition}[definition]{Proposition}

\newtheorem{theorem}[definition]{Theorem}
\newtheorem{corollary}[definition]{Corollary}

\theoremstyle{remark}

\numberwithin{equation}{section}

\renewcommand{\phi}{\varphi}
\newcounter{nootje}
\newcommand{\beq}{\begin{equation}}
\newcommand{\eeq}{\end{equation}}
\newcommand{\beqs}{\begin{equation*}}
\newcommand{\eeqs}{\end{equation*}}

\title{\vspace{-\baselineskip}\sffamily\bfseries A sharp upper bound for the $2$-torsion of class groups of multiquadratic fields}
\author[1]{Peter Koymans\thanks{Vivatsgasse 7, 53111 Bonn, Germany, koymans@mpim-bonn.mpg.de}}
\author[1]{Carlo Pagano\thanks{Vivatsgasse 7, 53111 Bonn, Germany, carlein90@gmail.com}}
\affil[1]{Max Planck Institute for Mathematics, Bonn}
\date{\today}

\begin{document}
\maketitle
	
\begin{abstract}
Let $K$ be a multiquadratic extension of $\Q$ and let $\text{Cl}^{+}(K)$ be its narrow class group. Recently, the authors \cite{KP} gave a bound for $|\text{Cl}^{+}(K)[2]|$ only in terms of the degree of $K$ and the number of ramifying primes. In the present work we show that this bound is sharp in a wide number of cases. Furthermore, we extend this to ray class groups.
\end{abstract}

\section{Introduction}
The class group is one of the most fundamental invariants of a number field $K$. Providing non-trivial upper bounds for the $l$-torsion of class groups in terms of the discriminant $\Delta_{K/\mathbb{Q}}$ of a general number field $K$ has been an active area of research with connections to elliptic curves and diophantine approximation \cite{BSTTTZ, EPW, FW, HV, Pierce, Pierce2, PTW, W}. 

For extensions $K/\mathbb{Q}$ of degree a power of a prime $l$ much more is known. For instance for $l=2$ and $K/\mathbb{Q}$ a quadratic extension, Gauss \cite{Gauss} showed that
$$
\text{dim}_{\mathbb{F}_2}\text{Cl}^{+}(K)[2]=\omega(\Delta_{K/\mathbb{Q}})-1. 
$$
Here $\text{Cl}^{+}(K)$ denotes the narrow class group of the field $K$ and $\omega(a)$ denotes the number of prime factors of a non-zero integer $a$. Recently, the authors \cite{KP} generalized Gauss' result to multiquadratic fields. More specifically, we obtained the following result, which is Theorem 1.1 of \cite{KP}. Call a vector $(a_1, \dots, a_n) \in \mathbb{Z}_{\geq 2}^n$ acceptable if the $a_i$ are squarefree, pairwise coprime and only have prime factors congruent to $1$ modulo $4$.

\begin{theorem} 
\label{tBound}
Let $n$ be a positive integer and let $(a_1, \dots, a_n) \in \mathbb{Z}_{\geq 2}^n$ be acceptable. Then we have
$$
\emph{dim}_{\mathbb{F}_2} \emph{Cl}^{+}(\mathbb{Q}(\sqrt{a_1}, \dots, \sqrt{a_n}))[2] \leq \omega(a_1 \cdot \ldots \cdot a_n) \cdot 2^{n - 1} - 2^n + 1. 
$$
\end{theorem}

A similar upper bound has subsequently been established by Kl\"uners and Wang in \cite[Theorem 2.1]{KW} for extensions $K/\mathbb{Q}$ of degree a power of $l$. However, when specialized to the multiquadratic fields considered above, their bound is in the worst case scenario twice as large as the one in Theorem \ref{tBound}. This work is devoted to showing that the bound in Theorem \ref{tBound} is sharp for every $n \in \mathbb{Z}_{\geq 1}$.

An acceptable vector $(a_1, \dots, a_n)$ is said to be \emph{maximal} in case the inequality of Theorem \ref{tBound} is an equality. Among other things, we have given a recursive characterization of maximal vectors (see \cite[Theorem 1.2]{KP}), which we reproduce now. Write $\pi_S$ for the projection on the coordinates in $S$, write $H_2^+(K)$ for the maximal multiquadratic unramified (at all finite places) extension of $K$ and write $[n] := \{1, \dots, n\}$.

\begin{theorem}
\label{tRecursive}
Let $n$ be a positive integer and let $(a_1, \dots, a_n)$ be an acceptable vector. Then the following are equivalent. \\
$(a)$ The vector $(a_1, \dots, a_n)$ is maximal, i.e. 
$$
\emph{dim}_{\mathbb{F}_2} \emph{Cl}^{+}(\mathbb{Q}(\sqrt{a_1}, \dots, \sqrt{a_n}))[2] = \omega(a_1 \cdot \ldots \cdot a_n) \cdot 2^{n - 1}-2^n+1.
$$
$(b)$ For every $j \in [n]$, the vector $\pi_{[n] - \{j\}}(a_1, \dots, a_n)$ is maximal and every prime divisor $p$ of $a_j$ splits completely in $H_2^{+}(\mathbb{Q}(\{\sqrt{a_m}\}_{m \in [n]-\{j\}}))$. \\
$(c)$ For every $j \in [n]$, the vector $\pi_{[n] - \{j\}}(a_1, \dots, a_n)$ is maximal and for every prime divisor $p$ of $a_j$, one (or equivalently any) prime above $p$ in the field $\mathbb{Q}(\{\sqrt{a_m}\}_{m \in [n]-\{j\}})$ belongs to  $2\emph{Cl}^{+}(\mathbb{Q}(\{\sqrt{a_m}\}_{m \in [n]-\{j\}}))$.  
\end{theorem}

In particular Theorem \ref{tRecursive} recovers the equality of Gauss' theorem for $n=1$ as a special case. It is then natural to ask whether for every positive integer $n$ one can find maximal vectors of dimension $n$. As the reader can sense from the characterization given in Theorem \ref{tRecursive}, it is not at all obvious how to do this. A naive inductive approach based on the Chebotarev Density Theorem runs into severe difficulties, since one needs to simultaneously guarantee splitting of a prime $p$ in a field $K_q$ depending on $q$ and of $q$ in a field $K_p$ depending on $p$.

To circumvent this problem, we use combinatorial ideas from \cite{Smith}, which we explain here from first principles in order to make the present work self-contained (see Section \ref{Additive Systems}). Our main theorem shows that one can find maximal vectors $(a_1, \ldots, a_n)$ for every $n$. Moreover, for any fixed $n$, we show that Theorem \ref{tBound} is sharp for a wide number of choices of $(\omega(a_1), \dots, \omega(a_n))$. More precisely, we establish the following.

\begin{theorem} 
\label{tSharp}
$(a)$ Take $n \in \Z_{> 3}$ and take $(k_1, \dots, k_n) \in \mathbb{Z}_{\geq 1} \times (2 \cdot \mathbb{Z}_{\geq 1})^{n-1}$. Then there are infinitely many acceptable vectors $(a_1, \ldots, a_n)$ with $\omega(a_i) = k_i$ for each $i \in [n]$ and
$$
\emph{dim}_{\mathbb{F}_2} \emph{Cl}^{+}(\mathbb{Q}(\sqrt{a_1}, \dots, \sqrt{a_n}))[2] = \omega(a_1 \cdot \ldots \cdot a_n) \cdot 2^{n - 1} - 2^n + 1. 
$$
$(b)$ Take $(k_1,k_2,k_3) \in \mathbb{Z}_{\geq 1}^3$. Then there are infinitely many acceptable vectors $(a_1, a_2, a_3)$ with $\omega(a_i) = k_i$ for each $i \in \{1, 2, 3\}$ and
$$
\emph{dim}_{\mathbb{F}_2} \emph{Cl}^{+}(\mathbb{Q}(\sqrt{a_1}, \sqrt{a_2}, \sqrt{a_3}))[2] = \omega(a_1a_2a_3) \cdot 4 - 7. 
$$
\end{theorem}

We speculate that the condition $(k_1, \dots, k_n) \in \mathbb{Z}_{\geq 1} \times (2 \cdot \mathbb{Z}_{\geq 1})^{n-1}$ can also be removed for $n>3$, but this seems to be out of reach with the techniques employed in this work. We next turn our attention to ray class groups. First of all, let us notice that the $2$-torsion of the ordinary class group of a number field $K$ can not be larger than the $2$-torsion of the narrow class group of $K$. Hence the upper bound in Theorem \ref{tBound} is also an upper bound for $|\text{Cl}(\mathbb{Q}(\sqrt{a_1}, \dots, \sqrt{a_n}))[2]|$. Less obvious is whether also this bound is sharp. 

Similarly, fix an integer $c$, which we take in this paper to be a squarefree product of primes congruent to $1$ modulo $4$ (see the end of this introduction for some motivation on this assumption). Then one obtains from Theorem \ref{tBound} and the ray class group sequence
$$
\text{dim}_{\mathbb{F}_2}\text{Cl}(\mathbb{Q}(\sqrt{a_1}, \dots, \sqrt{a_n}),c)[2] \leq \omega(a_1 \cdot \ldots \cdot a_n) \cdot 2^{n - 1} - 2^n + 1+2^n \cdot\omega(c),
$$
where the bound can be reached only if all the prime divisors of $c$ split completely in $\mathbb{Q}(\sqrt{a_1}, \ldots, \sqrt{a_n})$. It is, once more, not obvious whether this bound is sharp. Our next theorem settles these questions.  

\begin{theorem} 
\label{tSharp2}
Take $n \in \Z_{\geq 1}$ and take $(k_1, \dots, k_n) \in (2 \cdot \mathbb{Z}_{\geq 1})^{n}$. Let $c$ be a squarefree integer divisible only by primes congruent to $1$ modulo $4$. Then there are infinitely many acceptable vectors $(a_1, \ldots, a_n)$ with $\omega(a_i) = k_i$ for each $i \in [n]$ and
$$
\emph{dim}_{\mathbb{F}_2}\emph{Cl}(\mathbb{Q}(\sqrt{a_1}, \dots, \sqrt{a_n}),c)[2]= \omega(a_1 \cdot \ldots \cdot a_n) \cdot 2^{n - 1} - 2^n + 1+2^n \cdot\omega(c).
$$
\end{theorem}

As a corollary of Theorem \ref{tSharp2} we obtain the following result on unit groups. 

\begin{corollary}
\label{cUnits}
Let $n \in \Z_{\geq 1}$. Let $c$ be a squarefree integers with all factors congruent to $1$ modulo $4$. Then there exist infinitely many acceptable vectors $(a_1, \ldots, a_n)$ such that all prime divisors of $c$ split completely in $\mathbb{Q}(\sqrt{a_1}, \ldots, \sqrt{a_n})$ and the unit group $\mathcal{O}_{\mathbb{Q}(\sqrt{a_1}, \ldots, \sqrt{a_n})}^{*}$ reduced modulo $c$ is entirely contained in the group
$$
\left(\frac{\mathcal{O}_{\mathbb{Q}(\sqrt{a_1}, \ldots, \sqrt{a_n})}}{c}\right)^{*2}.
$$
\end{corollary}

We remark that, in the context of Corollary \ref{cUnits}, it is no real loss of generality to demand that all the prime divisors of $c$ are $1$ modulo $4$. Indeed, we are aiming to construct multiquadratic extensions splitting completely at all prime divisors of $c$ and whose unit group consists entirely of squares modulo $c$. This then in particular applies to $-1$, which is then a square in $\mathbb{F}_l$ for every $l \mid c$ so that $l \equiv 1 \bmod 4$. We similarly remark that the bound for the ordinary class group in Theorem \ref{tSharp2} (i.e. the case $c = 1$) is not sharp, whenever one of the $a_i$ is divisible by a prime congruent to $3$ modulo $4$.

\section*{Acknowledgments}
We thank Alexander Smith for several clarifying emails and conversations about his work. We are grateful to the anonymous referee of \cite{KP} for having encouraged us to write down a self-contained proof of Theorem \ref{tSharp} in the form of this independent work. The authors wish to thank the Max Planck Institute for Mathematics in Bonn for its financial support, great work conditions and an inspiring atmosphere. 

\section{Additive systems} 
\label{Additive Systems}
For completeness we include a self-contained proof of \cite[Proposition 3.1]{Smith}; we claim no originality in this section. 

We let $X_1, \dots, X_d$ be arbitrary non-empty finite sets and put $X := X_1 \times \dots \times X_d$. In our application the sets $X_i$ will consist of acceptable integers $a_i$ with $\omega(a_i) = k_i/2$. A cube $C$ is a product set $Y_1 \times \dots \times Y_d$ with $Y_i \subseteq X_i$ and $|Y_i| = 2$, in our application we can think of $C$ as an acceptable vector $(a_1, \dots, a_d)$ with $\omega(a_i) = k_i$. It is here that we make essential use that $k_i$ is even. As we see in our next section, we need to find cubes $C$ satisfying certain bilinear conditions. The aim of our next definition is to encapsulate this in an abstract framework. 

We write $X_i^2$ for the set $X_i \times X_i$. For $S \subseteq [d]$ and $i \in [d]$, $\pi_i$ denotes the natural projection from $\prod_{i \in S} X_i^2 \times \prod_{i \not \in S} X_i$ to $X_i^2$ if $i \in S$ and to $X_i$ if $i \not \in S$, while $\text{pr}_1$ and $\text{pr}_2$ denote the natural projections from $X_i^2$ to its two factors.

\begin{definition} 
\label{dAdditive}
Let $X_1, \dots, X_d$ be arbitrary non-empty finite sets and put $X := X_1 \times \dots \times X_d$. An additive system $\mathfrak{A}$ on $X$ is given by a tuple $(C_S, C_S^{\textup{acc}}, F_S, A_S)$ indexed by subsets $S \subseteq [d]$ satisfying the following properties
\begin{itemize}
\item $C_S^{\textup{acc}} \subseteq C_S \subseteq \prod_{i \in S} X_i^2 \times \prod_{i \not \in S} X_i$ are sets, $F_S: C_S \rightarrow A_S$ is a map and $A_S$ is a finite $\mathbb{F}_2$-vector space;
\item we have that 
$$
C_S^{\textup{acc}} := \left\{x \in C_S : F_S(x) = 0\right\}
$$
and for $S \neq \emptyset$
\begin{align*}
&C_S := \{x \in \prod_{i \in S} X_i^2 \times \prod_{i \not \in S} X_i : \textup{ for all } j \in S \textup{ and all } y \in \prod_{i \in S - \{j\}} X_i^2 \times \prod_{i \in [d] - (S - \{j\})} X_i \\
&\textup{ satisfying } \pi_k(x) = \pi_k(y) \textup{ for } k \in [d] - \{j\} \textup{ and } \pi_j(y) \in \{\textup{pr}_1(\pi_j(x)), \textup{pr}_2(\pi_j(x))\}, \\ \\
& \quad \quad \quad \quad \quad \quad \quad \quad \quad \quad \textup{ we have } y \in C_{S - \{j\}}^{\textup{acc}}\};
\end{align*}
\item suppose that $x_1, x_2, x_3 \in C_S$ and suppose that there exists $j \in S$ such that
\[
\pi_k(x_1) = \pi_k(x_2) = \pi_k(x_3) \textup{ for all } k \in [d] - \{j\}
\]
and
\[
\pi_j(x_1) = (a, b), \quad \pi_j(x_2) = (b, c), \quad \pi_j(x_3) = (a, c) \textup{ for some } a, b, c \in X_j.
\]
Then we have
\begin{align}
\label{eAdditivity}
F_S(x_1) + F_S(x_2) = F_S(x_3).
\end{align}
\end{itemize}
\end{definition}

Note that we do not quite work with cubes in the above definition, but instead with elements of $X_1 \times X_1 \times \dots \times X_d \times X_d$. The major difference is that we have also included elements with equal coordinates. This will be very convenient in the proof of our next counting result for $C_S^{\text{acc}}$. Later, we shall need to remove such elements, but it is not hard to show that they contribute a vanishingly small proportion.

\begin{proposition} 
\label{shrinking lemma}
Let $X_1, \dots, X_d$ be arbitrary non-empty finite sets and put $X := X_1 \times \dots \times X_d$. Let $\mathfrak{A}$ be an additive system on $X$ such that $|A_S| \leq a$ for all $S \subseteq [d]$ and write $\delta$ for the density of $C_\emptyset^{\textup{acc}}$ in $X$. Then we have that
\[
\frac{|C_{[d]}^{\textup{acc}}|}{\prod_{i \in [d]} |X_i^2|} \geq \delta^{2^d} \cdot a^{-3^d}.
\]
\end{proposition}

\begin{proof}
We proceed by induction on $d$ with the case $d = 0$ being trivial. Fix an element
\[
x \in \prod_{i \in [d - 1]} X_i.
\]
Let $V(x)$ be the subset of $a \in X_d$ such that $(x, a) \in C_{[d - 1]}^{\text{acc}}$ and let $W(x)$ be the subset of $(a, b) \in X_d^2$ such that $(x, (a, b)) \in C_{[d]}^{\text{acc}}$. By definition of an additive system, we see that $W(x)$ naturally injects in $V(x) \times V(x)$. From now on we shall identify $W(x)$ with its image in $V(x) \times V(x)$. We claim that $W(x)$ defines an equivalence relation on $V(x)$. 

If we apply equation (\ref{eAdditivity}) with $a = b = c$, we conclude that for all $S \subseteq [d - 1]$, all $y \in \prod_{i \in S} X_i \times \prod_{i \in [d - 1] - S} X_i$ and all $a \in X_d$
\[
F_{S \cup \{d\}}(y, (a, a)) = 0.
\]
From this, it follows quickly that $W(x)$ is reflexive. Applying equation (\ref{eAdditivity}) with $a = c$, we then get
\[
F_{S \cup \{d\}}(y, (a, b)) + F_{S \cup \{d\}}(y, (b, a)) = F_{S \cup \{d\}}(y, (a, a)) = 0,
\]
so that $W(x)$ is symmetric. Finally, equation (\ref{eAdditivity}) with $a$, $b$ and $c$ arbitrary implies the transitivity of $W(x)$, which establishes the claim.

Our next step is to estimate the number of equivalence classes. To do so, take $(x, a), (x, b) \in V(x)$ and $\{d\} \subseteq S \subseteq [d]$. Then we write $(x, a) \sim_S (x, b)$ if $(x, a) \sim_{S'} (x, b)$ for all $\{d\} \subseteq S' \subsetneq S$ and
\[
F_S(y, (a, b))) = 0
\]
for all $y \in \prod_{i \in S - \{d\}} X_i^2 \times \prod_{i \in [d - 1] - S} X_i$ satisfying $\pi_i(y) = \pi_i(x)$ for $i \in S - \{d\}$ and $\pi_i(y) \in \{\text{pr}_1(\pi_i(x)), \text{pr}_2(\pi_i(x))\}$ for $i \in [d - 1] - S$. Note that the equivalence relation $\sim_{[d]}$ is precisely $W(x)$. 

To upper bound the number of equivalence classes, take a collection of points $(x, a_i) \in V(x)$ such that $(x, a_i) \sim_S (x, a_j)$ for all strict subsets $S$ of $[d]$. Suppose that among this collection there are $R$ equivalence classes for $\sim_{[d]}$, with representatives $(x, b_1), \dots, (x, b_R)$. Then we see that the map
\[
(x, b_i) \mapsto F_{[d]}(x, (b_1, b_i))
\]
is injective and hence we conclude that $R \leq a$. If we proceed in this way, we see that the total number of equivalence classes for $\sim_{[d]}$ is bounded by
\[
\prod_{\{d\} \subseteq S \subseteq [d]} a^{2^{d - |S|}} \leq \prod_{i = 0}^{d - 1} a^{\binom{d - 1}{i} 2^i} = a^{3^{d - 1}},
\]
since for a given $S$, there are $2^{d - |S|}$ choices for $y$. Define
\[
\delta(x) = \frac{|V(x)|}{|X_d| \times \prod_{i \in [d - 1]} |X_i^2|}.
\]
Then it follows that the density of $V(x) \times V(x)$ in $\prod_{i \in [d]} X_i^2$ is $\delta(x)^2$. Then Cauchy's inequality and our bound for the total number of equivalence classes implies that
\[
\frac{|W(x)|}{\prod_{i \in [d]} |X_i^2|} \geq \frac{\delta(x)^2}{a^{3^{d - 1}}}.
\]
So far we have proven that
\[
\frac{|C_{[d]}^{\textup{acc}}|}{\prod_{i \in [d]} |X_i^2|} = \frac{\sum_{x \in \prod_{i \in [d - 1]} X_i^2} |W(x)|}{\prod_{i \in [d]} |X_i^2|} \geq \sum_{x \in \prod_{i \in [d - 1]} X_i^2} \frac{\delta(x)^2}{a^{3^{d - 1}}}.
\]
Another application of Cauchy's inequality shows that
\[
\sum_{x \in \prod_{i \in [d - 1]} X_i^2} \frac{\delta(x)^2}{a^{3^{d - 1}}} \geq \frac{\left(\sum_{x \in \prod_{i \in [d - 1]} X_i^2} \delta(x)\right)^2}{a^{3^{d - 1}} \cdot \prod_{i \in [d - 1]} |X_i^2|}.
\]
The average of $\delta(x)$ over all choices of $x$ equals the density of $C_{[d - 1]}^{\text{acc}}$ in $X_d \times \prod_{i \in [d - 1]} X_i^2$. The induction hypothesis yields
\[
\frac{\left(\sum_{x \in \prod_{i \in [d - 1]} X_i^2} \delta(x)\right)^2}{a^{3^{d - 1}} \cdot \prod_{i \in [d - 1]} |X_i^2|} \geq \frac{(\delta^{2^{d - 1}} \cdot a^{-3^{d - 1}})^2}{a^{3^{d - 1}}} = \delta^{2^d} \cdot a^{-3^d}
\]
as desired.
\end{proof}

\section{Proof of Theorem \ref{tSharp}}
In this section we prove Theorem \ref{tSharp}. The work is divided in two parts. In Subsection \ref{Preparations} we extract from \cite{KP} the basic results that will be needed in the proof, we prove Proposition \ref{prop: additivity} and we recall a version of R\'edei reciprocity, Theorem \ref{prop: redei reciprocity}, that will be used later. With these tools in hand, we give the proof of Theorem \ref{tSharp} in Subsection \ref{Proof}. 

\subsection{Preparations} 
\label{Preparations}
The shape of Theorem \ref{tRecursive} presents a striking resemblance with Definition \ref{dAdditive}. To make the analogy more stringent one would like to turn the splitting conditions in part $(b)$ of Theorem \ref{tRecursive} into an \emph{additive system}: this is precisely the route we are going to follow. To do so we recall a refinement of Theorem \ref{tRecursive}, which will invoke the language of \emph{expansion maps}. We now recall the definition from \cite[Section 3.3]{KP}. If $A$ is a set, we write $\FF_2^A$ for the free $\FF_2$-vector space on $A$.

\begin{definition}
Let $G$ be a profinite group and let $A \subseteq \textup{Hom}(G, \FF_2)$ be a finite, linearly independent set with $|A| \geq 2$ and $\chi_0 \in A$. An expansion map for $G$ with support $A$ and pointer $\chi_0$ is a continuous group homomorphism
\[
\psi : G \rightarrow \FF_2[\FF_2^{A - \{\chi_0\}}] \rtimes \FF_2^{A - \{\chi_0\}}
\]
satisfying the following two properties
\begin{itemize}
\item for every $\chi \in A - \{\chi_0\}$, we have $\pi_\chi \circ \psi = \chi$, where $\pi_\chi$ is the projection on the coordinate of $\chi$ in $\FF_2^{A - \{\chi_0\}}$;
\item we have $\widetilde{\chi} \circ \psi = \chi_0$, where $\widetilde{\chi}$ is the unique non-trivial character of $\FF_2[\FF_2^{A - \{\chi_0\}}] \rtimes \FF_2^{A - \{\chi_0\}}$ that sends the subgroup $\{0\} \rtimes \FF_2^{A - \{\chi_0\}}$ to $0$.
\end{itemize}
\end{definition}

If $\psi$ is an expansion map for $G_{\mathbb{Q}}$, we define its \emph{field of definition} to be $L(\psi) := \overline{\mathbb{Q}}^{\text{ker}(\psi)}$. Denote by $\chi_a$ the character corresponding to $\Q(\sqrt{a})$.

\begin{theorem}
\label{t:all expansions}
Let $n$ be a positive integer and let $(a_1, \dots, a_n)$ be an acceptable vector. Then the following are equivalent. \\
$(a)$ The vector $(a_1, \dots, a_n)$ is maximal, i.e. 
$$
\emph{dim}_{\mathbb{F}_2} \emph{Cl}^{+}(\mathbb{Q}(\sqrt{a_1}, \dots, \sqrt{a_n}))[2] = \omega(a_1 \cdot \ldots \cdot a_n) \cdot 2^{n - 1}-2^n+1.
$$
$(b)$ For every $T \subsetneq [n]$, every $j \in [n] - T$ and every prime $p \mid a_j$, there exists an expansion map
$$
\psi_{T,p}: \emph{Gal}(H_2^{+}(\mathbb{Q}(\{\sqrt{a_h} : h \in T\} \cup \{\sqrt{p}\}))/\mathbb{Q}) \to \mathbb{F}_2[\mathbb{F}_2^T] \rtimes \mathbb{F}_2^T
$$
with support $\{\chi_{a_h}\}_{h \in T} \cup \{\chi_p\}$ and pointer $\chi_p$.
 
Furthermore, in case one of the two equivalent statements $(a),(b)$ holds, then the set of expansion maps described in $(b)$ when restricted to the group 
$$
\emph{Gal}(H_2^{+}(\mathbb{Q}(\sqrt{a_1}, \ldots, \sqrt{a_n}))/\mathbb{Q}(\sqrt{a_1}, \ldots, \sqrt{a_n}))
$$
gives a generating set for $\emph{Cl}^{+}(\mathbb{Q}(\sqrt{a_1}, \dots, \sqrt{a_n}))^{\vee}[2]$.
\end{theorem}

\begin{proof}
This follows from \cite[Theorem 3.20]{KP} and \cite[Proposition 4.1]{KP}.
\end{proof}

We shall need further understanding of expansion maps, and to this end we recall some more material from \cite[Section 3.3]{KP}. Let $e_i$ be the $i$-th basis vector of $\mathbb{F}_2^T$, which we can naturally view as an element of $\mathbb{F}_2[\mathbb{F}_2^T]$. There is a ring isomorphism
\[
\mathbb{F}_2[\mathbb{F}_2^T] \cong \FF_2[t_1, \dots, t_n]/(t_1^2, \dots, t_n^2)
\]
by sending $t_i$ to $1 \cdot \text{id} + 1 \cdot e_i$. Under this isomorphism, the action of $e_i \in \FF_2^T$ becomes multiplication by $1 + t_i$. If $\psi_{T, p}$ is an expansion map, then projection on the monomials $t_S := \prod_{i \in S} t_i$ gives a system of $1$-cochains
$$
\phi_S(\psi_{T, p}): \text{Gal}(H_2^{+}(\mathbb{Q}(\{\sqrt{a_h} : h \in T\} \cup \{\sqrt{p}\}))/\mathbb{Q}) \to \mathbb{F}_2
$$
for each $S \subseteq T$. These $1$-cochains satisfy the recursive equation
\begin{align}
\label{e2.1}
\phi_S(\sigma \tau) - \phi_S(\sigma) - \phi_S(\tau) = \sum_{\emptyset \neq U \subseteq S} \chi_U(\sigma) \phi_{S - U}(\tau)
\end{align}
with $\phi_\emptyset = \chi_p$ and $\chi_U = \prod_{i \in U} \chi_{a_i}$, where the product is taken in $\FF_2$. Reversely, a system of $1$-cochains satisfying equation (\ref{e2.1}) naturally gives rise to an expansion map. Next, a vector
$$
(\psi_{T - \{i\}, p})_{i \in T}
$$
of expansion maps with supports $\{\chi_{a_j}\}_{j \in T - \{i\}} \cup \{\chi_p\}$ and pointer $\chi_p$ for each $i \in T$, is called a \emph{commutative vector} in case for every $i, j \in T$
$$
\phi_{T - \{i, j\}}(\psi_{T - \{i\}, p}) = \phi_{T - \{i, j\}}(\psi_{T - \{j\}, p}).
$$
Note that Theorem \ref{tRecursive} implies that a maximal vector $(a_1, \ldots, a_n)$ must be \emph{strongly quadratically consistent}, i.e. we have $(\frac{p}{q}) = 1$ for every distinct $i, j \in [n]$ and every two primes $p \mid a_i, q \mid a_j$. 

\begin{theorem} 
\label{tCommvect}
Let $n$ be a positive integer, and let $(a_1, \ldots, a_n)$ be an acceptable vector, which is strongly quadratically consistent. Let $T \subsetneq [n]$, let $j \in [n]-T$ and let $p$ be a prime divisor of $a_j$. Then the following are equivalent. \\
$(a)$  There exists an expansion map
$$
\psi_{T,p}: \emph{Gal}(H_2^{+}(\mathbb{Q}(\{\sqrt{a_h} : h \in T\} \cup \{\sqrt{p}\}))/\mathbb{Q}) \to \mathbb{F}_2[\mathbb{F}_2^T] \rtimes \mathbb{F}_2^T
$$
with support $\{\chi_{a_h}\}_{h \in T} \cup \{\chi_p\}$ and pointer $\chi_p$. \\
$(b)$ There exists a commutative vector of expansion maps 
$$
(\psi_{T-\{i\},p})_{i \in T}
$$
with supports $\{\chi_{a_h}\}_{h \in T-\{i\}} \cup \{\chi_p\}$ and pointer $\chi_p$ for each $i \in T$, satsfying the following condition. For every $i \in T$ and every prime divisor $q$ of $a_i$, we have that $q$ splits completely in the field of definition of $\psi_{T-\{i\},p}$.
\end{theorem}

\begin{proof}
This is a special case of \cite[Theorem 1.5]{KP}.
\end{proof}

In order to prove part (a) of Theorem \ref{tSharp}, we aim to combine Theorem \ref{t:all expansions} and Theorem \ref{tCommvect} with Proposition \ref{shrinking lemma}. An import stepping stone is to guarantee equation (\ref{eAdditivity}) for the various cochains $\phi_S(\psi_{T,p})$ attached to an expansion map $\psi_{T,p}$. We now explain what this means and how to achieve this. 

Let $n \in \Z_{\geq 1}$, let $(k_1, \ldots, k_n) \in \mathbb{Z}_{\geq 1} \times (2 \cdot \mathbb{Z}_{\geq 1})^{n-1}$ and let $M \in \mathbb{Z}_{\geq 1}$. Take
$$
Y := Y_1 \times \dots \times Y_n
$$
to be a product space, where each $Y_i$ is a set of cardinality $M$ consisting of acceptable squarefree integers. We further require that any two distinct elements in $\cup_{i = 1}^n Y_i$ are pairwise coprime and that $\omega(z) = \frac{k_i}{2}$ for each $i \in [n] - \{1\}$ and $z \in Y_i$, while $\omega(z) = k_1$ for $z \in Y_1$. We call such a $Y$ a $((k_1, \ldots,k_n),M)$-space.

Let $Y$ now be a $((k_1, \ldots,k_n),M)$-space. We denote by $K(Y)$ the multiquadratic number field obtained by adding all the square roots of the prime divisors of the elements in $\cup_{i=1}^{n}Y_i$ to $\Q$. Observe that for each prime $p$ ramifying in $K(Y)/\mathbb{Q}$, the inertia subgroups of $p$ in $\text{Gal}(H_2^{+}(K(Y))/\mathbb{Q})$ are cyclic of size $2$. For each such prime $p$ we fix a choice of such an inertia element $\sigma_p$. We will denote this choice by $\mathfrak{G}:=\{\sigma_p\}_{p \mid \prod_{i=1}^{n}(\prod_{y \in Y_i}y)}$ and refer to it as \emph{a choice of inertia} for $Y$. 

\begin{proposition} 
\label{prop: at most one expansion}
$(a)$ Let $Y$ be a $((k_1, \ldots,k_n),M)$-space together with a choice of inertia $\mathfrak{G}$. Let $S \subsetneq [n]$ and let $j \in [n] - S$. Pick a non-trivial divisor $d$ of an element in $Y_j$ and pick $\{a_i\}_{i \in S}$ with $a_i$ a product of elements in $Y_i$ for each $i \in S$. Then there exists at most one expansion map
$$
\psi_{(a_i)_{i \in S}; d}(\mathfrak{G}): \emph{Gal}(H_2^+(K(Y))/\mathbb{Q}) \to \mathbb{F}_2[\mathbb{F}_2^{\{\chi_{a_i}: i \in S \ \textup{and} \ \chi_{a_i} \neq 0\}}] \rtimes \mathbb{F}_2^{\{\chi_{a_i}: i \in S \ \textup{and} \ \chi_{a_i} \neq 0\}},
$$
with support $\{\chi_{a_i} : i \in S \textup{ and } \chi_{a_i} \neq 0\} \cup \{\chi_d\}$ and pointer $\chi_d$ such that
$$
\phi_T(\psi_{(a_i)_{i \in S}; d}(\mathfrak{G}))(\sigma) = 0
$$
for each $\emptyset \neq T \subseteq S$ and each $\sigma \in \mathfrak{G}$. \\
$(b)$ If $\psi_{(a_i)_{i \in S}; d}(\mathfrak{G})$ exists, then it factors through $\emph{Gal}(H_2^{+}(\mathbb{Q}(\{\sqrt{a_i}\}_{i \in S}, \sqrt{d}))/\mathbb{Q})$.
\end{proposition}
 
\begin{proof}
Since $\mathbb{Q}$ has no non-trivial unramified extensions, the group $\text{Gal}(H_2^+(K(Y))/\mathbb{Q})$ is generated by the conjugacy classes of all elements in $\mathfrak{G}$. We claim that $\mathfrak{G}$ already generates $\text{Gal}(H_2^{+}(K(Y))/\mathbb{Q})$. Indeed, if $G$ is any finite group and $S \subseteq G$, then $S$ generates $G$ if and only if $S$ generates $G/\Phi(G)$, where $\Phi(G)$ is the Frattini subgroup. Furthermore, for a $2$-group we know that the Frattini subgroup $\Phi(G)$ equals $G^2 [G, G]$, so that two conjugate elements have the same image in $G/\Phi(G)$. This gives part $(a)$ immediately, since the requirement $\phi_T(\psi_{(a_i)_{i \in S}; d}(\mathfrak{G}))(\sigma) = 0$ for each $\emptyset \neq T \subseteq S$ determines the image of $\sigma$ under $\psi_{(a_i)_{i \in S}; d}(\mathfrak{G})$. 

To obtain part $(b)$ we start by noticing that $L(\psi_{(a_i)_{i \in S}; d}(\mathfrak{G}))$ is an abelian extension of $\mathbb{Q}(\{\sqrt{a_i}\}_{i \in S}, \sqrt{d})$. We only need to guarantee that it is unramified at all finite places. For this it is enough to notice that for each prime $q$ not dividing $a_i$ nor $d$ one has that 
$$
\psi_{(a_i)_{i \in S}; d}(\mathfrak{G})(\sigma_q) = \text{id},
$$
precisely thanks to our requirement that $\phi_T(\psi_{(a_i)_{i \in S}; d}(\mathfrak{G}))(\sigma_q) = 0$ for each $\emptyset \neq T \subseteq S$. 
\end{proof}

The next proposition gives the sought behavior among expansion maps. For convenience we introduce the following notation. Let $S \subseteq [n]$ and let $U \subseteq [n]-S$. Let 
$$
x \in \prod_{i \in S} Y_i^2 \times \prod_{j \in U} Y_j,
$$
then we write
$$
c(x) := ((\text{pr}_1(\pi_i(x)) \text{pr}_2(\pi_i(x)))_{i \in S}, (\pi_j(x))_{j \in U})
$$
for the vector obtained by multiplying out the double entries of $x$ and leaving unchanged the single entries of $x$. 

\begin{proposition}
\label{prop: additivity}
Let $Y$ be a $((k_1, \ldots,k_n),M)$-space together with a choice of inertia $\mathfrak{G}$. Let $S \subsetneq [n]$, let $j \in [n] - S$ and $i_0 \in S$. Pick a non-trivial divisor $d$ of an element in $Y_j$. Let $U \subseteq [n]-S-\{j\}$. Let $x_1, x_2, x_3$ be three elements of $\prod_{i \in S} Y_i^2 \times \prod_{u \in U}Y_u$ such that they coincide outside $i_0$ and such that 
$$
\textup{pr}_1(\pi_{i_0}(x_1)) = \textup{pr}_2(\pi_{i_0}(x_3)), \ \ \  \textup{pr}_1(\pi_{i_0}(x_2)) = \textup{pr}_2(\pi_{i_0}(x_1)), \ \ \  \textup{pr}_1(\pi_{i_0}(x_3)) = \textup{pr}_2(\pi_{i_0}(x_2)).
$$
Suppose $\psi_{c(x_1); d}(\mathfrak{G})$ and $\psi_{c(x_2); d}(\mathfrak{G})$ exist. Then the map $\psi_{c(x_3); d}(\mathfrak{G})$ exists and
$$
\phi_T(\psi_{c(x_3); d}(\mathfrak{G})) = \phi_T(\psi_{c(x_1); d}(\mathfrak{G}))+\phi_T(\psi_{c(x_2); d}(\mathfrak{G}))
$$
for each $\emptyset \neq T \subseteq S$.
\end{proposition}

\begin{proof}
This is now an immediate consequence of Proposion \ref{prop: at most one expansion}. Indeed, the maps 
$$
\phi_T(\psi_{(\pi_i(c(x_1)))_{i \in S}; d}(\mathfrak{G})) + \phi_T(\psi_{(\pi_i(c(x_2)))_{i \in S}; d}(\mathfrak{G}))
$$
yield an expansion map from the group $\text{Gal}(H_2^{+}(K(Y))/\mathbb{Q})$ to the group 
$$
\mathbb{F}_2[\mathbb{F}_2^{\{\chi_{\pi_i(c(x_3))}: i \in S \cup U \ \text{and} \ \chi_{\pi_i(c(x_3))} \neq 0\}}] \rtimes \mathbb{F}_2^{\{\chi_{\pi_i(c(x_3))}: i \in S \cup U \ \text{and} \ \chi_{\pi_i(c(x_3))} \neq 0\}}.
$$
Furthermore, the vanishing at all elements of $\mathfrak{G}$ follows by construction. This gives the desired conclusion. 
\end{proof}  

We next give a more specific version of Theorem \ref{tCommvect} that encodes the choice of inertia elements $\mathfrak{G}$. We call a $((k_1, \ldots,k_n),M)$-space $Y$ \emph{quadratically consistent} in case each of its vectors are strongly quadratically consistent. 

\begin{theorem}
\label{tCommvect-with-inertia}
Let $Y$ be a quadratically consistent $((k_1, \ldots,k_n),M)$-space, together with a choice of inertia $\mathfrak{G}$. Let $S \subsetneq [n]$ and let $j \in [n] - S$. Pick a non-trivial divisor $d$ of an element in $Y_j$. Pick furthermore $U \subseteq [n] - S - \{j\}$. Let $a$ be an element of $\prod_{i \in S} Y_i^2 \times \prod_{u \in U} Y_u$. Then the following are equivalent. \\
$(a)$ The map $\psi_{c(a); d}(\mathfrak{G})$ exists. \\
$(b)$ For each $h \in S \cup U$ the map $\psi_{\pi_{S \cup U-\{h\}}(c(a)); d}(\mathfrak{G})$ exists and every prime ramifying in $\mathbb{Q}(\sqrt{a_h})/\mathbb{Q}$ splits completely in the field of definition of $\psi_{\pi_{S \cup U-\{h\}}(c(a)); d}(\mathfrak{G})$
\end{theorem}

\begin{proof}
Proposition \ref{prop: at most one expansion} shows that the vector $(\psi_{(\pi_i(c(a)))_{i \in S \cup U: i \neq h}; d}(\mathfrak{G}))_{h \in S \cup U}$ is commutative. Hence the conclusion follows from Theorem \ref{tCommvect}, provided that we can ensure that $\varphi_{S \cup U}$ vanishes on $\mathfrak{G}$. But, looking at equation (\ref{e2.1}), we see that we still have the freedom to twist $\varphi_{S \cup U}$ by the characters $\chi_p$ for $p$ ramifying in $K(Y)/\mathbb{Q}$.
\end{proof}

Finally, in order to obtain Theorem \ref{tSharp}, part $(b)$, we recast here (a special case of) R\'edei reciprocity, re-written in the language of expansion maps. Suppose that $(a_1,a_2,a_3)$ is a strongly quadratically consistent vector. Then there exists an expansion map $\psi_{a_1; a_2}: G_{\mathbb{Q}} \to \mathbb{F}_2[\mathbb{F}_2] \rtimes \mathbb{F}_2$ such that every prime divisor $p$ of $a_3$ splits completely in $\mathbb{Q}(\sqrt{a_1}, \sqrt{a_2})/\mathbb{Q}$.

Hence $\text{Frob}(p)$ lands in the central subgroup $\text{Gal}(L(\psi_{a_1; a_2})/\mathbb{Q}(\sqrt{a_1}, \sqrt{a_2}))$, which can be canonically identified with $\mathbb{F}_2$: here we recall that $L(\psi_{a_1;a_2})$ denotes the field of definition of an expansion map. In what follows Frobenius symbols need to be interpreted as elements of $\mathbb{F}_2$.

\begin{theorem} 
\label{prop: redei reciprocity}
Let $(a_1,a_2,a_3)$ be a strongly quadratically consistent vector. Let $\psi_{a_1; a_2}: \emph{Gal}(H_2^+(\mathbb{Q}(\sqrt{a_1}, \sqrt{a_2}))/\mathbb{Q}) \to \mathbb{F}_2[\mathbb{F}_2] \rtimes \mathbb{F}_2$ and $ \psi_{a_1; a_3}: \emph{Gal}(H_2^+(\mathbb{Q}(\sqrt{a_1}, \sqrt{a_3}))/\mathbb{Q}) \to \mathbb{F}_2[\mathbb{F}_2] \rtimes \mathbb{F}_2$ be expansion maps with supports respectively $\{\chi_{a_1}, \chi_{a_2}\}$ and $\{\chi_{a_1}, \chi_{a_3}\}$ and pointers respectively $\chi_{a_2}$ and $\chi_{a_3}$. Then
$$
\sum_{p \mid a_3} \emph{Frob}_{L(\psi_{a_1; a_2})/\mathbb{Q}}(p) = \sum_{p \mid a_2} \emph{Frob}_{L(\psi_{a_1; a_3})/\mathbb{Q}}(p).
$$
\end{theorem}

\begin{proof}
This is a special case of \cite[Theorem 3.3]{KPconics}.
\end{proof}

\begin{remark} 
\label{HRR}
Theorem \ref{prop: redei reciprocity} has recently been generalized by the authors to more general expansion maps, see \cite[Theorem 3.3]{KPconics}. It is natural to wonder if this reciprocity law allows one to generalize the proof of Theorem \ref{tSharp} part $(b)$ to $n > 3$. For every $(k_1, \ldots, k_n)$ we have been able to construct vectors $(a_1, \ldots, a_n)$ with $(\omega(a_1), \ldots, \omega(a_n)) = (k_1, \ldots, k_n)$ and $|\textup{Cl}^{+}(\mathbb{Q}(\sqrt{a_1}, \ldots, \sqrt{a_n}))[2]|$ ``large". However, already for $n = 4$, we have not been able to produce maximal vectors this way.
\end{remark}

\subsection{Proof of Theorem \ref{tSharp}} 
\label{Proof}
Let us start with a proposition that immediately yields part $(b)$ and will be an important step for part $(a)$. 

\begin{proposition} 
\label{prop: a Redei-started space}
Let $N$ and $m$ be positive integers. Then there exists a product space 
$$
X := X_1 \times \dots \times X_m,
$$
where the $X_i$ are disjoint sets of primes congruent to $1$ modulo $4$ with $|X_i| = N$ for each $i \in [m]$ such that
$$
\prod_{i \in S} X_i
$$
consists entirely of maximal vectors for every subset $S \subseteq [m]$ with $1 \leq |S| \leq 3$.
\end{proposition}

\begin{proof}
We proceed by induction on $m$. For $m = 1$ the statement is trivial. Now suppose that the statement is true for $m$, so that we have to prove it for $m + 1$. Pick a product set $X_1 \times \dots \times X_m$ guaranteed by the inductive hypothesis. Consider the set $Z$ of primes that split completely in $H_2^{+}(K(X_1 \times \dots \times X_m)) \Q(\sqrt{-1})/ \mathbb{Q}$. Thanks to the Chebotarev Density Theorem, we see that $Z$ is an infinite set. 

Pick any $N$-set $X_{m + 1}$ inside $Z$. Observe that $X_{m + 1}$ is disjoint from each of the $X_i$ with $i \leq m$, since these are all primes ramifying in $H_2^{+}(K(X_1 \times \dots \times X_m))/ \mathbb{Q}$. Next, since $X_{m+1}$ consists in particular of primes splitting in $K(X_1 \times \dots \times X_m)/ \mathbb{Q}$, we see that $X_{i_1} \times X_{i_2}$ consists entirely of maximal vectors for every distinct $i_1$ and $i_2$. Hence for each $2$-set $\{i_1, i_2\}$ and every point $(p, q) \in X_{i_1} \times X_{i_2}$ we have an expansion map
$$
\psi_{p; q}: \text{Gal}(H_2^+(\mathbb{Q}(\sqrt{p},\sqrt{q}))/\mathbb{Q}) \to \mathbb{F}_2[\mathbb{F}_2] \rtimes \mathbb{F}_2
$$
with support $\{\chi_p, \chi_q\}$ and pointer $\chi_q$. 

Thanks to our choice of $X_{m + 1}$, we have that every $x$ in $X_{m + 1}$ splits completely in $L(\psi_{p; q})$ whenever $i_1, i_2 \leq m$ are distinct. But then Theorem \ref{prop: redei reciprocity} yields that $p$ splits completely in $L(\psi_{q; x})$ and $q$ splits completely in $L(\psi_{p; x})$. Therefore the proposition follows from Theorem \ref{t:all expansions} and Theorem \ref{tCommvect}.
\end{proof}

\begin{proof}[Proof of Theorem \ref{tSharp} part $(b)$] 
By taking $m = 3$ and $N$ arbitrary large, we see that Proposition \ref{prop: a Redei-started space} immediately implies part $(b)$ of Theorem \ref{tSharp} for $(k_1, k_2, k_3) = (1, 1, 1)$. The general case then follows from Proposition \ref{prop: additivity}.
\end{proof}

\begin{proof}[Proof of Theorem \ref{tSharp} part $(a)$] 
Take $n \in \mathbb{Z}_{\geq 4}$ and $(k_1, \ldots, k_n) \in \mathbb{Z}_{\geq 1} \times (2\mathbb{Z}_{\geq 1})^{n-1}$. Fix furthermore an auxiliary parameter $M \in \mathbb{Z}_{\geq 1}$. It follows from Proposition \ref{prop: additivity} and Proposition \ref{prop: a Redei-started space} that we can construct a $((k_1, \ldots, k_n), M)$-space 
$$
Y := Y_1 \times \dots \times Y_n,
$$
equipped with a choice of inertia $\mathfrak{G}$ such that for any $3$-set $\{i_1,i_2,i_3\} \subseteq [n]$, any triple $(y_{i_1},y_{i_2},y_{i_3}) \in Y_{i_1} \times Y_{i_2} \times Y_{i_3}$ and any prime divisor $p \mid y_{i_3}$ we have that the map
$$
\psi_{y_{i_1},y_{i_2};p}(\mathfrak{G})
$$
exists. Fix such a $((k_1, \ldots, k_n), M)$-space $Y$. Also fix a point $y_1 \in Y_1$ and put 
$$
\widetilde{Y}:=\{y_1\} \times Y_2 \times \ldots \times Y_n.
$$
We are going to construct an additive system on $\widetilde{Y}$ for the subsets $S$ of $[n] - \{1\}$. We start by defining subsets
$$
C_S \subseteq \prod_{i \in S} Y_i^2 \times \prod_{j \in [n] - S} Y_j
$$
for every $S \subseteq [n] - \{1\}$. Let us first consider the case that $|S| \leq n-2$. We define $C_S$ by the property that $a \in C_S$ if and only if for each $2$-set $\{i, j\} \subseteq [n] - S$ and for every prime divisor $p \mid \pi_j(a)$, we have that
$$
\psi_{\pi_{S \cup \{i\}}(c(a)); p}(\mathfrak{G})
$$ 
exists. We next put $C_{[n]-\{1\}}$ to be the set of $a$ in $\{y_1\} \times \prod_{2 \leq h \leq n}Y_h^2$ such that for any $j \in [n] - \{1\}$ and any prime divisor $p$ of $\pi_j(c(a))$ we have that
$$
\psi_{\pi_{[n]-\{j\}}(c(a)); p}(\mathfrak{G})
$$
exists and furthermore
\[
\psi_{\pi_{[n]-\{1, j\}}(c(a)), \text{pr}_k(\pi_j(a)); p}(\mathfrak{G})
\]
exists for all $j \in [n] - \{1\}$, $k \in [2]$ and $p$ dividing $y_1$.

We now define the spaces $A_S$. Assume first that $|S| \leq n-3$. We put $A_S$ to be the space of formal $\mathbb{F}_2$-linear combinations of $5$-tuples
$$
(x_1, x_2, x_3, x_4, x_5),
$$ 
where $x_1,x_2,x_3 \in [n] - S$ are pairwise distinct and $x_4 \in [k_{x_2}], x_5 \in [k_{x_3}]$. Instead for $|S| \in \{n - 2, n - 1\}$, we set $A_S = \{0\}$. 

Let us now define $F_S: C_S \to A_S$. In case $|S| > n - 3$, we set $F_S$ to be the trivial map. Henceforth we assume that $|S| \leq n - 3$. Let $(x_1, x_2, x_3, x_4, x_5)$ be a $5$-tuple as above and $a \in C_S$. Let $p_{x_4}(a)$ be the $x_4$-th prime divisor of $\pi_{x_2}(c(a))$, by the natural ordering, and let $p_{x_5}(a)$ be the $x_5$-th prime divisor of $\pi_{x_3}(c(a))$. We have that the Frobenius of $p_{x_5}(a)$ lands in the center of 
$$
\text{Gal}(L(\psi_{\pi_{S \cup \{x_1\}}(c(a)); p_{x_4}(a)}(\mathfrak{G}))/\mathbb{Q})
$$
thanks to Theorem \ref{tCommvect-with-inertia} and the definition of $C_S$. Observe that the center of
\[
\FF_2[t_1, \dots, t_n]/(t_1^2, \dots, t_n^2) \rtimes \FF_2^n
\] 
is cyclic of order $2$ and generated by $t_1 \cdot \ldots \cdot t_n$. Hence to decide whether an element of the center is trivial or not one may simply apply the $1$-cochain $\phi_{S \cup \{x_1\}}(\psi_{\pi_{S \cup \{x_1\}; p_{x_4}(a)}}(\mathfrak{G}))$ to the central element. In other words the value
$$
\phi_{S \cup \{x_1\}}(\psi_{\pi_{S \cup \{x_1\}}(c(a)); p_{x_4}(a)}(\mathfrak{G}))(\text{Frob}(p_{x_5}(a)))
$$
is well-defined and equals $0$ if and only if $p_{x_5}(a)$ splits completely in the field of definition of $\psi_{\pi_{S \cup \{x_1\}}(c(a)); p_{x_4}(a)}(\mathfrak{G})$. With this preliminary in mind, we define $F_S(a)$ to be the vector of $A_S$ whose $(x_1, x_2, x_3, x_4, x_5)$-coordinate equals 
$$
\phi_{S \cup \{x_1\}}(\psi_{\pi_{S \cup \{x_1\}}(c(a)); p_{x_4}(a)}(\mathfrak{G}))(\text{Frob}(p_{x_5}(a)))
$$
for each $5$-tuple $(x_1,x_2,x_3,x_4,x_5)$ as described above. Finally, for each $S \subseteq [n]-\{1\}$, we put
$$
C_S^{\text{acc}} := F_S^{-1}(0).
$$
We now establish the following crucial fact.  

\begin{proposition} 
\label{prop: construction of additive system}
The $4$-tuple $\{(C_S, C_S^{\emph{acc}}, F_S, A_S)\}_{S \subseteq [n] - \{1\}}$ defined above is an additive system on $\widetilde{Y}$. Furthermore,
$$
|A_S| \leq 2^{n \cdot (\sum_{i = 1}^n k_i)^2}
$$
for each $S \subseteq [n] - \{1\}$. 

Finally, for all $a \in C_{[n]-\{1\}}$ we have that the vector $(\pi_i(c(a)))_{i \in [n]: \chi_{\pi_i(c(a))} \neq 0}$ is a maximal vector of dimension $|\{i \in [n]: \chi_{\pi_i(c(a))} \neq 0\}|$.  
\end{proposition}

\begin{proof}
Equation (\ref{eAdditivity}) is satisfied thanks to Proposition \ref{prop: additivity}. The bound on $|A_S|$ follows from straightforward counting. The maximality claim is a consequence of Theorem \ref{t:all expansions} and Theorem \ref{tCommvect}.
\end{proof}

We now finish the proof of Theorem \ref{tSharp}, part $(a)$. Due to Proposition \ref{prop: construction of additive system} and Proposition \ref{shrinking lemma} we deduce that there exists a positive number $c_{(k_1, \ldots, k_n)}$, depending only on the vector $(k_1, \ldots, k_n)$, such that there are at least 
$$
c_{(k_1, \ldots, k_n)} \cdot M^{2n-2}
$$
vectors $a \in \{y_1\} \times \prod_{2 \leq i \leq n}Y_i^2$ with $(\pi_i(c(a)))_{i \in [n]: \chi_{\pi_i(c(a))} \neq 0}$ maximal. On the other hand, no more than $(n-1)\cdot M^{2n-3}$ vectors $a$ in $\{y_1\} \times \prod_{2 \leq i \leq n}Y_i^2$ are such that $\text{pr}_1(\pi_i(a)) = \text{pr}_2(\pi_i(a))$ for some $i$. It follows that there at least 
$$
c_{(k_1, \ldots, k_n)} \cdot M^{2n-2} - (n - 1) \cdot M^{2n-3}
$$
vectors in $C_{[n] - \{1\}}$ with distinct coordinates. Each of them gives a maximal vector $c(a)$ such that
$$
\omega(\pi_i(c(a))) = k_i
$$
for each $i \in [n]$. Precisely $2^{n - 1}$ choices of $a$ will give rise to the same vector when passing to $c(a)$. All in all we have obtained at least
$$
\frac{c_{(k_1, \ldots, k_n)} \cdot M^{2n-2} - (n - 1) \cdot M^{2n-3}}{2^{n-1}}
$$
distinct multiquadratic fields $\mathbb{Q}(\sqrt{a_1}, \ldots, \sqrt{a_n})$ with $(a_1, \ldots, a_n)$ a maximal vector and with $\omega(a_i) = k_i$ for each $i \in [n]$. For $M$ going to infinity this quantity goes to infinity, which gives us the desired conclusion.
\end{proof}
  
\section{Proof of Theorem \ref{tSharp2} and Corollary \ref{cUnits}}
In this section we give a proof of Theorem \ref{tSharp2} and Corollary \ref{cUnits}. We start by demonstrating that Corollary \ref{cUnits} is a simple consequence of Theorem \ref{tSharp2}. Denote by $K := \mathbb{Q}(\sqrt{a_1}, \ldots, \sqrt{a_n})$ a field satisfying the conclusion of Theorem \ref{tSharp2}. Recall that we have an exact sequence
$$
0 \to \frac{(\mathcal{O}_K/c)^{*}}{\mathcal{O}_K^{*}} \to \text{Cl}(K,c) \to \text{Cl}(K) \to 0. 
$$
To ease the notation, let us denote by $A$ the group $\frac{(\mathcal{O}_K/c)^{*}}{\mathcal{O}_K^{*}}$. This gives the inequality
\begin{align*}
\text{dim}_{\mathbb{F}_2} \text{Cl}(K,c)[2] 
&\leq \text{dim}_{\mathbb{F}_2} \text{Cl}(K)[2] + \text{dim}_{\mathbb{F}_2}A[2] \\
&\leq \omega(a_1 \cdot \ldots \cdot a_n) \cdot 2^{n-1} - 2^n + 1 + 2^n\cdot \omega(c). 
\end{align*}
The second inequality can be an equality only if
$$
\text{dim}_{\mathbb{F}_2}\text{Cl}(K)[2] = \omega(a_1 \cdot \ldots \cdot a_n) \cdot 2^{n-1} -2^n + 1,
$$
and
$$
\text{dim}_{\mathbb{F}_2} A[2] = 2^n \cdot \omega(c),
$$
thanks to Theorem \ref{tBound} (for the first equation) and simple counting (for the second equation). Therefore we deduce from
$$
\text{dim}_{\mathbb{F}_2} \text{Cl}(K,c)[2] =\omega(a_1 \cdot \ldots \cdot a_n) \cdot 2^{n-1} - 2^n + 1 + 2^n \cdot \omega(c)
$$
that
$$
\text{dim}_{\mathbb{F}_2} \text{Cl}(K)[2] = \omega(a_1 \cdot \ldots \cdot a_n) \cdot 2^{n-1} - 2^n + 1
$$
and
$$
\text{dim}_{\mathbb{F}_2} A[2] = 2^n\cdot \omega(c).
$$
Observe that we have a surjection
$$
\phi:\frac{(\mathcal{O}_K/c)^{*}}{(\mathcal{O}_K/c)^{*2}} \to \frac{A}{2A}.
$$
The above equations imply that
$$
\text{dim}_{\mathbb{F}_2} \frac{(\mathcal{O}_K/c)^{*}}{(\mathcal{O}_K/c)^{*2}} = 2^n \cdot \omega(c) = \text{dim}_{\mathbb{F}_2}A[2] = \text{dim}_{\mathbb{F}_2}\frac{A}{2A},
$$
whence $\phi$ is an isomorphism. On the other hand 
$$
\text{ker}(\phi) = \text{im}\left(\text{red}_c(K): \frac{\mathcal{O}_K^{*}}{\mathcal{O}_K^{*2}} \to \frac{(\mathcal{O}_K/c)^{*}}{(\mathcal{O}_K/c)^{*2}}\right),
$$
where $\text{red}_c(K)$ is the natural reduction map modulo $c$. We conclude that the map $\text{red}_c(K)$ is trivial as desired.

It remains to prove Theorem \ref{tSharp2}. To this end we switch to the set-up of the proof of Theorem \ref{tSharp} and indicate the necessary modifications. First of all, we recall that the choice of $y_1$ was arbitrary, so we are allowed to take $y_1 := c$. Now suppose that $(c, a_1, \ldots, a_n)$ is a maximal vector such that all expansion maps have totally real field of definition. Then we claim that 
$$
\text{dim}_{\mathbb{F}_2} \text{Cl}(\mathbb{Q}(\sqrt{a_1}, \ldots, \sqrt{a_n}), c)[2] = \omega(a_1 \cdot \ldots \cdot a_n)\cdot 2^{n-1} - 2^n + 1 + 2^{n} \cdot \omega(c).
$$
Surely we have that
$$
\text{dim}_{\mathbb{F}_2} \text{Cl}(\mathbb{Q}(\sqrt{a_1}, \ldots, \sqrt{a_n}))[2] = \omega(a_1 \cdot \ldots \cdot a_n) \cdot 2^{n-1} - 2^n + 1.
$$
But now observe that the collection of characters
$$
\{\phi_{T}(\psi_{a_1, \ldots, a_n; l}(\mathfrak{G}))\}_{l \mid c \ \text{prime}, \ T \subseteq [n]}
$$
is linearly independent and generates a subspace of
$$
\text{Cl}(\mathbb{Q}(\sqrt{a_1}, \ldots, \sqrt{a_n}), c)^{\vee}[2]
$$
linearly disjoint from 
$$
\text{Cl}(\mathbb{Q}(\sqrt{a_1}, \ldots, \sqrt{a_n}))^{\vee}[2]
$$
by ramification considerations. This gives precisely the $2^n \cdot \omega(c)$ additional characters in $\text{Cl}(\mathbb{Q}(\sqrt{a_1}, \ldots, \sqrt{a_n}),c)^{\vee}[2]$ and therefore yields
$$
\text{dim}_{\mathbb{F}_2} \text{Cl}(\mathbb{Q}(\sqrt{a_1}, \ldots, \sqrt{a_n}), c)[2] = \omega(a_1 \cdot \ldots \cdot a_n) \cdot 2^{n-1} - 2^n + 1 + 2^n \cdot \omega(c)
$$
as desired.

We still need to explain how one ensures that all expansion maps are totally real. First of all, we indicate how Proposition \ref{prop: a Redei-started space} can be modified to ensure that all the maps $\psi_{y_1; y_2}(\mathfrak{G})$ are totally real. In this case we use a more general version \cite{Stevenhagen} of R\'edei reciprocity, which includes $-1$ (taking the role of the infinite place). Next one enlarges $A_S$ to encode the splitting condition at infinity, and the maps $F_S$ are also extended accordingly. With these modifications in mind, one proceeds exactly with the same argument as in Theorem \ref{tSharp}.

\end{document}